\documentclass[11pt,a4paper]{article}
\usepackage[T1]{fontenc}
\usepackage[utf8]{inputenc}
\usepackage{hyperref}
\usepackage{mathtools}
\usepackage{amsfonts}
\usepackage{amsthm}
\usepackage{amssymb}
\usepackage[shortlabels]{enumitem}
\usepackage{graphicx}
\usepackage[dvipsnames]{xcolor}
\usepackage{keyval}
\usepackage{tikz}
\usepackage{tikz-cd}
\usepackage{blindtext}
\usepackage{xurl}

\newtheorem{theorem}{Theorem}[section]
\newtheorem{lemma}[theorem]{Lemma}
\newtheorem{corollary}[theorem]{Corollary}
\newtheorem{fact}[theorem]{Fact}
\theoremstyle{definition}
\newtheorem{definition}[theorem]{Definition}

\newcommand{\id}{\mathrm{id}}

\newcommand{\pushout}[4]{\begin{tikzcd}[ampersand replacement=\&] #1 \arrow[r]\arrow[d] \& #2 \arrow[d]\\ #3 \arrow[r] \& #4 \arrow[ul, pos=0, phantom, "\ulcorner"] \end{tikzcd}}
\newcommand{\hq}{\mathbin{/\!/}}

\newcommand{\Type}{\mathcal{U}}
\newcommand{\refl}{\operatorname{refl}}
\newcommand{\Prop}{\operatorname{Prop}}

\newcommand{\Contr}{\operatorname{Contr}}
\newcommand{\proj}{\mathrm{proj}}

\newcommand{\inr}{\operatorname{inr}}
\newcommand{\inl}{\operatorname{inl}}
\newcommand{\inx}{\operatorname{in}}
\newcommand{\tr}[2][]{\left\lVert #2\right\rVert_{#1}}

\newcommand{\ap}{\mathit{ap}}
\newcommand{\glue}{\operatorname{glue}}

\newcommand{\SP}{\operatorname{SP}}
\newcommand{\qtwo}{q}
\newcommand{\stwo}{s}
\newcommand{\Shor}{\psi}
\newcommand{\SPhor}{\varphi}
\newcommand{\Sinl}{\alpha}
\newcommand{\Sinr}{\beta}
\newcommand{\SPinl}{\gamma}
\newcommand{\SPinr}{\epsilon}
\newcommand{\col}{\operatorname{col}}
\newcommand{\colAt}{\operatorname{colAt}}
\newcommand{\cb}{\operatorname{cb}}
\newcommand{\ColR}{\operatorname{ColR}}
\newcommand{\rhobar}{\rho}
\newcommand{\Lam}{\Lambda}
\newcommand{\LL}{L}
\newcommand{\WW}{W}

\usetikzlibrary{arrows.meta,positioning}

\usepackage{geometry}
\title{\bfseries The third symmetric product of a set is a set in HoTT}
\author{Wojciech Paupa}
\date{September 2026}

\begin{document}

\maketitle

\begin{abstract}
    We deformalize an LLM-generated proof within the formalism of Homotopy Type Theory that an iterated pushout construction for $\SP^3(X)$ is a set whenever $X$ is a set. The proof expands on a similar proof by Buchholtz for $\SP^2(X)$, employing a similar encode-decode strategy and the same formal machinery.
\end{abstract}

\section*{Introduction}
Homotopy Type Theory (HoTT) is a foundation for mathematics based around Martin-L\"of Type Theory extended with the univalence axiom and higher inductive types, as described in \cite{book}. Types in HoTT can be interpreted as logical formulae or as $\omega$-groupoids, which in the homotopy interpretation correspond to certain topological spaces, up to homotopy equivalence. Thus, many homotopy-invariant constructions defined on such spaces have their direct analogues in constructions defined on homotopy types. The symmetric products, despite being homotopy invariant, are defined as strict colimits, and so their HoTT analogue is not a priori impossible to define, but also not directly translatable. In \cite{upairs}, a pushout definition for the second symmetric product was proposed, and it was proven that it maps sets to sets. A sketch of the definition of the third symmetric product was also provided, and it was hypothesized that it also maps sets to sets. In this work, we provide a pushout definition of the third symmetric product and prove that it satisfies that claim.

The version of HoTT we work in is the one described in \cite{book} (often called Book HoTT), with higher inductive types, function and pair extensionality (implied by univalence), a hierarchy of univalent universes, and without a cubical interval type.

The proof of the main theorem (theorem \ref{t:sp3set}) was generated and formalized using an Arend + Claude Fable toolchain, and then later deformalized in this paper.
\section{Methodology}
The formal proofs of the theorems in this paper are available at \url{https://github.com/WPaupa/symmetric-products}, with the proof of theorem \ref{t:sp3set} located in \verb|src/Experiments/SP3|. The work refers to commit \verb|9866e8c|.
\subsection{Arend}
\href{https://github.com/arend-lang/Arend}{Arend} is a type-theoretic proof assistant developed at JetBrains Research with the aim of providing native support for constructive and univalent mathematics (see \cite{arend}). Its foundation is homotopy type theory with an interval type equipped with computational axioms and a built-in way of defining functions by path induction. Unlike the cubical type theory, Arend does not satisfy full canonicity for natural numbers. However, the homotopical features of related theories, like path types, higher inductive types, and univalence, are present in the language. Arend's features also include universes of truncated types, code inference, and metaprogramming, all of which proved relevant to the project.

The version of Arend used in this project is 1.11, and syntax used in the repository is currently incompatible with Arend 1.12.
\subsection{LLM Toolchain}
The proof was generated and formalized using an agentic loop between Claude Fable 5 in the chat interface, Claude Opus 5 running Claude Code, and the Arend standalone compiler, as illustrated in figure \ref{fig:agentic-loop}. The information flow between the models was manual.
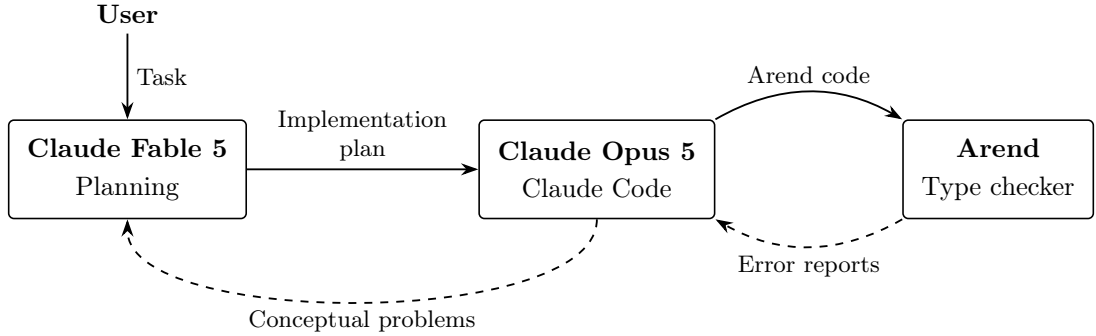
\begin{figure}[!h]
  \centering
  \begin{tikzpicture}[
    font=\small,
    agent/.style={
      draw, rounded corners=2pt, line width=0.6pt,
      minimum height=1.3cm, minimum width=3cm,
      align=center, inner sep=7pt
    },
    flow/.style={
      -{Stealth[length=2.3mm,width=1.6mm]},
      line width=0.7pt
    },
    feedback/.style={flow, dashed},
    label/.style={
      font=\footnotesize, align=center, inner sep=3pt
    }
  ]
    % Participants.
    \node (user) at (0,2.05) {\textbf{User}};
    \node[agent] (planner) at (0,0)
      {\textbf{Claude Fable 5}\\[2pt]Planning};
    \node[agent] (coder) at (6.2,0)
      {\textbf{Claude Opus 5}\\[2pt]Claude Code};
    \node[agent, minimum width=2.2cm] (arend) at (11.5,0)
      {\textbf{Arend}\\[2pt]Type checker};

    % Task and implementation plan.
    \draw[flow] (user.south) --
      node[label, right] {Task} (planner.north);
    \draw[flow] (planner.east) --
      node[label, above] {Implementation\\plan} (coder.west);

    % Implementation and type-checker feedback.
    \draw[flow] (coder.north east)
      to[out=30,in=150]
      node[label, above] {Arend code} (arend.north west);
    \draw[feedback] (arend.south west)
      to[out=210,in=-30]
      node[label, below] {Error reports} (coder.south east);

    % Conceptual feedback to the planning agent.
    \draw[feedback] (coder.south)
      .. controls +(0,-1.45) and +(0,-1.45) ..
      node[label, below] {Conceptual problems} (planner.south);
  \end{tikzpicture}
  \caption{Proof formalization workflow. Dashed arrows indicate feedback.}
  \label{fig:agentic-loop}
\end{figure}

The planning agent was a fresh instance without custom instructions, while the coding agent used a set of helper files (called skills) published by Sergey Sinchuk at \url{https://github.com/sxhya/arend-skills}. It worked in conjunction with command-line tools for LLM agents available within the Arend compiler, including search for theorem patterns, search for definition usage, and diagnostics in a machine-readable format. With these tools the agents are less likely to examine files manually, which speeds up the workflow and reduces token usage.
\subsection{Deformalization}
Since the idea behind the proof was refined during formalization, its main points were sometimes obfuscated. This led to unused definitions and functions whose names didn't match their final roles. One additional LLM pass was used after finishing the proof to refactor the code and make it more human-readable, which helped to isolate the main lemmas. The deformalization process then consisted of going through the statements of the lemmas and reconstructing their proofs by hand, referring back to the formalized code whenever difficulties arose.
\section{Conventions and preliminaries}
Throughout the article, $\equiv$ is used for judgmental equality (or equality up to computation), and $=$ is used for propositional equality. For example, if we say that $a=b$ for $a,b:A$, we mean that the type $a=_Ab$ is inhabited. We use $\tr{A}$ for the propositional truncation. We write $\mathbf{n}$ for the $n$-point type, here using its inductive definition with $\mathbf{0}$ defined as the empty type, $\mathbf{1}$ defined as a type given by a single point constructor $*:\mathbf{1}$, and $\mathbf{m}=\mathbf{n}\sqcup\mathbf{1}$ for $m=n+1$. Finally, we write $\Type$ for a universe, omitting its level whenever not relevant. 
\subsection{Symmetric products}
\begin{definition}
    In classical algebraic topology, the $n$-th symmetric product of a topological space $X$ is the following quotient space:
    \begin{equation}
        \SP^nX:\equiv X^n/\Sigma_n\equiv X^n/(x_1,\dots,x_n)\sim(x_{\sigma(1)},\dots,x_{\sigma(n)}),\sigma\in\Sigma_n
    \end{equation}
    In other words, $\SP^nX$ is the space of unordered $n$-tuples of elements of $X$ with a quotient topology inherited from $X^n$.
\end{definition}
This construction is used for example in the Dold-Thom theorem (see section 4.K in \cite{hatcher}), a classic result bridging the theory of homology and homotopy groups.
\subsection{Pushouts}
\begin{definition}
    In Homotopy Type Theory, a pushout of $f:A\to B$ and $g:A\to C$ written as:
    \begin{equation}
        \begin{tikzcd}
            A \arrow[r, "f"] \arrow[d, "g"] & B \arrow[d] \\
            C \arrow[r] & D \arrow[ul, pos=0, phantom, "\ulcorner"]
        \end{tikzcd}
    \end{equation}
    is the higher inductive type $D$ given by constructors:
    \begin{equation}
        \begin{cases}
            \inx_B:B\to D\\
            \inx_C:C\to D\\
            \glue:\prod_{a:A}\inx_B(f(a))=\inx_C(g(a))
        \end{cases}.
    \end{equation}
\end{definition}
The join $A*B$ is the pushout of $\proj_1:A\times B\to A$ and $\proj_2:A\times B\to B$. For the point constructors of the join, we will write $\inl:A\to A*B$ and $\inr:B\to A*B$ instead of $\inx_A$ and $\inx_B$, the notation slightly conflicting with the notation for the disjoint union $A\sqcup B$. The mapping cone of $f:A\to B$ is the pushout $C$ of $f$ and the unique function $A\to\mathbf{1}$. Instead of $\inx_Bb,\inx_{\mathbf{1}}*$, we will just write $\inx b,*$ whenever unambiguous.

To define a (dependent) function out of a pushout, we will use pushout induction, which means defining it on the point constructors $\inx_B,\inx_C$ and then proving a (dependent) identification between its values lying over the relevant $\glue$ path. When the codomain of the function is a proposition, we can omit the identification.
\begin{fact}
    \label{f:joincontr}
    If $A$ or $B$ is contractible, then so is $A*B$. If $B$ is a proposition and $a:A$, then $A*B$ is contractible if and only if there exists an identification $\prod_{a':A}\inx_Aa=_{A*B}\inx_Aa'$. 
\end{fact}
\subsection{Borel product}
\begin{definition}
    The type $B\Sigma_n$ is the type of (small) types merely equivalent to $\mathbf{n}$:
    \begin{equation}
        B\Sigma_n:\equiv\sum_{X:\Type}\tr{X\simeq \mathbf{n}},
    \end{equation}
    In other words, $B\Sigma_n$ is the type of unordered $n$-element types. We then define the $n$-th Borel product of $X$ as:
    \begin{equation}
        X^n\hq \Sigma_n:\equiv\sum_{A:B\Sigma_n}(\proj_1A\to X)
    \end{equation}
    or the type of tuples over $X$ indexed by unordered $n$-element types.
\end{definition}
We will identify elements of $B\Sigma_n$ with the $n$-element types representing them, and use set algebra on them whenever it is unambiguous to do so. For example, for $K:B\Sigma_n$ and $K':B\Sigma_m$, $K\sqcup K':B\Sigma_{n+m}$ is the unique element $K''$ of $B\Sigma_{n+m}$ such that $\proj_1K''=\proj_1K\sqcup\proj_1K'$. For $(K,f):X^n\hq \Sigma_n$ and $(K',f'):X^m\hq \Sigma_m$ we will then write $(K\sqcup K',f\sqcup f'):X^{n+m}\hq \Sigma_{n+m}$. We will also write $k:K$ instead of $k:\proj_1K$ for $K:B\Sigma_n$. Finally, we may write the labelling function as $(x_0,\dots,x_{n-1})$ when the indexing type is $\mathbf{n}$.

For $(K,f),(K',f'):X^n\hq \Sigma_n$, a path $(K,f)=_{X^n\hq \Sigma_n}(K',f')$ consists of an equivalence $e:K\simeq K'$ and an identification $\prod_{k:K}f(k)=f'(e(k))$. For short, a path in $X^n\hq \Sigma_n$ is an equivalence commuting with the labelling functions.

\begin{fact}
    The type $B\Sigma_n$ is connected and if $X$ is connected, then so is $X^n\hq \Sigma_n$. 
\end{fact}
\section{The second symmetric product}
The results of this section come directly from \cite{upairs}.
\begin{definition}
    For any $X:\Type$, define $\SP^2X$ as the pushout:
    \begin{equation}
        \begin{tikzcd}
            B\Sigma_2\times X \arrow[r, "\delta"] \arrow[d, "\mathrm{proj}_X"] & X^2\hq \Sigma_2 \arrow[d, "\qtwo"]\\
            X \arrow[r, "\stwo"] & \SP^2X \arrow[ul, pos=0, phantom, "\ulcorner"]
        \end{tikzcd},\quad\delta(b,x)=(b,\lambda\_.x)
    \end{equation}
\end{definition}
\begin{theorem}[Buchholtz]
    \label{t:sp2set}
    If $X$ is a set, then $\SP^2X$ is also a set.
\end{theorem}
We recall the proof to motivate our approach to the $\SP^3$ case.
\begin{proof}
    It suffices to show that loop spaces in $\SP^2X$ are contractible for a model point $z_0$ such that $z_0=\qtwo(w_0)$ and $w_0=(\mathbf{2},(x_0,x_1))$ for $x_0,x_1: X$. Indeed, by a pushout induction, for each $x:\SP^2X$ there merely exists a path between $x$ and a model point, and since $\Contr$ is a proposition, we can then prove that the loop space based at $x$ is contractible.

    We use an encode-decode strategy through a family $Q:\SP^2X\to\Type$, with the goal of proving that $Q(z)$ is the type $z_0=z$, and then proving contractibility of $Q(z_0)$. We set $P:\equiv x_0=x_1$ and:
    \begin{equation}
        \begin{cases}
            Q(\qtwo(K,f))=(w_0=(K,f))*(P\times(f=\lambda\_.x_0)),\\
            Q(\stwo(x))=P\times (x=x_0)
        \end{cases}
    \end{equation}
    To finish this definition, we need to define the glue case, $\ap_Q(\glue(K,x)):Q(\qtwo(K,\lambda\_.x))=_\Type Q(\stwo(x))$ for $K:B\Sigma_2$ and $x:X$. We will prove that both sides of the equality are propositions. For the right side it is easy, for the left side we assume it has an element and prove it is contractible. If $w_0=(K,\lambda\_.x)$, then $x_0=x_1=x$, so any inhabitant of the join implies $P\times(\lambda\_.x=\lambda\_.x_0)$, which contracts the join. 

    We will now prove that $D:\equiv \sum_{z:\SP^2(X)}Q(z)$ is contractible. By the fundamental theorem of identity types (5.8.2 in \cite{book}), this will give us the desired relation $Q(z)\simeq(z_0=z)$. By the flattening lemma (6.12.2 in \cite{book}), we get a pushout diagram:
    \begin{equation}
        \pushout{\sum_{v:B\Sigma_2\times X}P\times (\proj_2v=x_0)}{\sum_{w:X^2\hq \Sigma_2}Q(\qtwo(w))}{\sum_{x:X}P\times (x=x_0)}{D}
    \end{equation}
    Now we can use the fact that $\sum_{s:S}(s_0=s)$ is contractible for any $S:\Type,s_0:S$ (lemma 3.11.8 in \cite{book}) and commutativity of sigma-types to simplify the corners of the pushout. We get the diagram:
    \begin{equation}
    \label{eq:pushoutD}
        \begin{tikzcd}
            B\Sigma_2\times P \arrow[r, "\proj_2"] \arrow[d] & P \arrow[d]\\
            \sum_{w:X^2\hq \Sigma_2}Q(\qtwo(w)) \arrow[r] & D \arrow[ul, pos=0, phantom, "\ulcorner"]
        \end{tikzcd}
    \end{equation}
    where the vertical map is $(K,p)\mapsto((K,\lambda\_.x_0),\inr(p,\refl))$. Denote the lower left corner as $R$, by commutativity of sigma-types and pushouts (and the same contractibility argument as above) we get: 
    \begin{equation}
    \label{eq:pushoutR}
        \pushout{P}{B\Sigma_2\times P}{\mathbf{1}}{R}
    \end{equation}
    where the horizontal map is $p\mapsto(\mathbf{2},p)$. 

    To finish the proof of contractibility of $D$, we show that $d=\inx_R(\inx_{\mathbf1}(*))$ for any $d:D$ by double induction on the pushout of \eqref{eq:pushoutD} and of \eqref{eq:pushoutR}.
    \begin{itemize}
        \item For the $p:P$ case, we have $\glue(\mathbf{2},p):\inx_Pp=\inx_R(\inx_{B\Sigma_2\times P}(\mathbf{2},p))$ and $\ap_{\inx_R}(\glue p):\linebreak\inx_R(\inx_{B\Sigma_2\times P}(\mathbf{2},p))=\allowbreak\inx_R(\inx_{\mathbf1}(*))$, composing the two paths gives the desired path.
        \item For the $R$ case, we proceed by a second layer of induction.
        \begin{itemize}
            \item For the $(K,p):B\Sigma_2\times P$ case, we have $\glue_D(K,p)^{-1}:\inx_R(\inx_{B\Sigma_2\times P}(K,p))=\inx_Pp$, $\glue_D(\mathbf{2},p):\inx_Pp=\inx_R(\inx_{B\Sigma_2\times P}(\mathbf{2},p))$ and $\ap_{\inx_{R}}(\glue_R p):\inx_R(\inx_{B\Sigma_2\times P}(\mathbf{2},p))=\linebreak\inx_R(\inx_{\mathbf1}(*))$, composing the three paths gives the desired paths.
            \item For the $\mathbf{1}$ case, the path is $\refl$.
            \item For the glue case, we have to prove an equality of form $(\ap_f(r_2))^{-1}=r_1^{-1}\cdot r_1\cdot\ap_f(r_2^{-1})$, which is true by path induction.
        \end{itemize}
        \item The glue case follows by commutativity of action on paths and the path inverse.
    \end{itemize}

    Now all that is left is to prove that $Q(z_0)$ is contractible. Since the right side is a proposition, it suffices to prove that $\inl\refl_{w_0}=_{Q(z_0)}\inl p$ for any $p:w_0=w_0$. Now we can treat an inhabitant of $w_0=w_0$ as an equivalence $\mathbf{2}\simeq \mathbf{2}$ commuting with the labelling functions. If this equivalence is the identity, $p$ is exactly $\refl_{w_0}$. If it is the unique non-identity equivalence (the swap), by the commutation relation we get $x_0=x_1$, which inhabits $P$ and contracts the join.
\end{proof}
\section{The third symmetric product}
This definition of the third symmetric product has been adapted from \cite{upairs}.
\begin{definition}
\label{d:SP3}
For any $X:\Type$, we define $\SP^3X$ as the following iterated pushout:
\begin{equation}
    \begin{tikzcd}
        B\Sigma_2\times X^2 \arrow[r, "\Shor"] \arrow[d, "\mathrm{proj}_{X^2}"] & X^3\hq \Sigma_3 \arrow[d, "\Sinl"] &
        X\times (B\Sigma_3\hq B\Sigma_2)\arrow[r,"\SPhor"] \arrow[d, "\mathrm{proj}_X"] & S^3_2(X) \arrow[d, "\SPinl"]
        \\ X^2 \arrow[r, "\Sinr"] &  \arrow[ul, pos=0, phantom, "\ulcorner"] S^3_2(X) & 
        X \arrow[r, "\SPinr"] & \SP^3(X) \arrow[ul, pos=0, phantom, "\ulcorner"]
    \end{tikzcd},
\end{equation}
where $B\Sigma_3\hq B\Sigma_2$ is the mapping cone of the inclusion $K\mapsto K\sqcup\mathbf{1}$, $\Shor$ is the map $(K,x,y)\mapsto(K\sqcup\mathbf{1},(\lambda\_.x)\sqcup(\lambda\_.y))$, and $\SPhor$ is a map defined by:
\begin{equation}
    \begin{cases}
        \SPhor(x,\inx t)&=\Sinl(t,\lambda\_.x)\\
        \SPhor(x,*)&=\Sinr(x,x)\\
        \ap_{\SPhor(x,\cdot)}(\glue_{B\Sigma_3\hq B\Sigma_2} K) &=\ap_\Sinl(\theta_K^x)\cdot\glue_{S^3_2(X)}(K,(x,x))
    \end{cases}
\end{equation}
with $\theta_K^x:(K\sqcup\mathbf1,\lambda\_.x)=(K\sqcup\mathbf1,\lambda\_.x\sqcup\lambda\_.x)$ being the natural identification.
\end{definition}
\begin{theorem}
    \label{t:sp3set}
    If $X$ is a set, then $\SP^3X$ is also a set.
\end{theorem}
Once again, it suffices to show that loop spaces in $\SP^3X$ are contractible for a model point $z_0$, this time we take $z_0:\equiv\SPinl(y_0)$, $y_0:\equiv\Sinl(w_0)$ and $w_0:\equiv(\mathbf{3},(x_0,x_1,x_2))$ for some $x_0,x_1,x_2$. Indeed, by a double pushout induction, for each $x: \SP^3X$ there merely exists a path between $x$ and a model point, and since $\Contr$ is a proposition, we can then prove that the loop space based at $x$ is contractible. For the following lemmas, we fix a set $X$, elements $x_0,x_1,x_2:X$, and use the notations of $z_0,y_0,w_0$. 

For $\SP^2$, we had two cases to handle: the diagonal and the off-diagonal. As seen in the definition of $\SP^3$, the additional case we will have here is a collision case, also known as the fat diagonal, where one coordinate appears twice. Apart from the diagonal proposition $P$, we define additional propositions for identifying the collision case.
\begin{definition}
    Let $\col:X^2\to X^3\hq \Sigma_3$ be defined as $\col(x,y):\equiv\Shor(\mathbf{2},x,y)$. Then we define the proposition $P:\Prop$ and collision propositions $\colAt:X^2\to\Prop$, $\cb:X^3\hq \Sigma_3\to\Prop$ as:
    \begin{align}
        P&:\equiv(x_0=x_1)\times(x_1=x_2),\\
        \colAt(a,b)&:\equiv\tr{w_0=\col(a,b)},\\
        \cb(u)&:\equiv\tr{\sum_{a,b:X}(w_0=\col(a,b))\times(u=\col(a,b))}
    \end{align}
    We read $\cb(u)$ as ``$u$ and $w_0$ are in the same collision case''. 
\end{definition}
\begin{lemma}
    \label{l:coleq}
    If $\col(a,b)=\col(a',b')$, then $a=a'$ and $b=b'$. 
\end{lemma}
\begin{proof}
    The path $\col(a,b)=\col(a',b')$ in $X^3\hq \Sigma_3$ is an equivalence $e:\mathbf{2}\sqcup\mathbf{1}\simeq\mathbf{2}\sqcup\mathbf{1}$ commuting with the labelling functions. If $e(\inr*)=\inr*$, then $e$ sends $\inl\cdot$ to $\inl\cdot$, and by its commutation we get $b=b'$ and $a=a'$. If $e(\inr*)=\inl y$ for some $y:\mathbf{2}$, we get $b=a'$. However, there exists an $x:\mathbf{2}$ such that $e(\inl x)=\inr*$, which gives $a=b'$, and $e(\inl(\neg x))=\inl(\neg y)$, which gives $a=a'$. 
\end{proof}
\begin{corollary}
\label{c:ColR}
The type $\ColR:\equiv\sum_{a,b:X}\colAt(a,b)$ is a proposition. Furthermore, given a fixed $(a,b,c):\ColR$ and for all $u:X^3\hq \Sigma_3$, we have:
\begin{equation}
    \label{eq:w0col}
    \tr{u=\col(a,b)}\leftrightarrow\tr{w_0=u}
\end{equation}
Therefore:
\begin{equation}
    \left(\sum_{u:X^3\hq \Sigma_3}\cb(u)\right)\simeq\left(\ColR\times\sum_{u:X^3\hq \Sigma_3}\tr{w_0=u}\right)
\end{equation}
\end{corollary}
The proposition $\ColR$ is the analogue of $P$ for the collision case.  
\begin{proof}
    Let $(a,b,c),(a',b',c'):\ColR$. Then composing $\tr{w_0=\col(a,b)}$ and $\tr{w_0=\col(a',b')}$ we get $\lVert\col(a,b)=\allowbreak \col(a',b')\rVert$, which by lemma \ref{l:coleq} gives us $(a=a')\times (b=b')$. Then $(a,b,c)=(a',b',c')$ follows. For $(a,b,c):\ColR$, both directions of the equivalence \eqref{eq:w0col} follow from composing with $c:\tr{w_0=\col(a,b)}$.
\end{proof}
We now employ the same encode-decode strategy as with $Q$ for $\SP^2X$. However, since $\SP^3X$ is defined by a double pushout, we will introduce two type families, one defined on the intermediate step, and the other defined on the actual space. Once again, our intention is to make $Q(z)$ encode the paths $z_0=z$. 
\begin{definition}
    \label{d:E}
    Define $J:X^3\hq \Sigma_3\to\Prop$ as $J(u):\equiv\cb(u)*(P\times (\proj_2u=\lambda\_.x_0))$. We can define an $E:S^3_2(X)\to\Type$ that has:
    \begin{equation}
        \begin{cases}
            E(\Sinl(u)) &:\equiv (w_0=u)*J(u),\\
            E(\Sinr(a,b)) &:\equiv \colAt(a,b)*(P\times a=x_0\times b=x_0).
        \end{cases}
    \end{equation}
\end{definition}
\begin{proof}
     To finish the definition, we need to prove that $E(\Sinl(\Shor(K,x,y)))=E(\Sinr(x,y))$. The right side is a proposition as a join of propositions, we will prove that so is the left. For an inhabitant $p:E(\Sinl(\Shor(K,x,y)))$, we get an inhabitant of $J(\Shor(K,x,y))$ in the right and glue case, and a path $w_0=\Shor(K,x,y)$ in the left case, which by connectedness of $B\Sigma_2$ gives us $\cb(\Shor(K,x,y))$ and $J(\Shor(K,x,y))$. Thus $J(\Shor(K,x,y))$ is then contractible, and so is $E(\Sinl(\Shor(K,x,y)))$. Now to prove $E(\Sinl(\Shor(K,x,y)))\leftrightarrow E(\Sinr(x,y))$ we need four functions on legs of the pushouts, all easy to define. 
\end{proof}
\begin{definition}
    We can define a $Q:\SP^3X\to\Type$ that has:
    \begin{equation}
        \begin{cases}
            Q(\SPinl(v))&:\equiv E(v),\\
            Q(\SPinr(a))&:\equiv P\times (a=x_0).
        \end{cases}
    \end{equation}
    and let $D:\equiv\sum_{z:\SP^3(X)}Q(z)$.
\end{definition}
\begin{proof}
    To finish the definition, we need to define the glue case for $Q$. For $(a,c):X\times B\Sigma_3\hq B\Sigma_2$, $Q(\SPinl(\SPhor(a,c)))$ and $Q(\SPinr(a))$ are both propositions, the first one by mapping cone induction on $c$. Now we need functions $E(\SPhor(a,c))\leftrightarrow P\times (a=x_0)$. By mapping cone induction on $c$ for both directions, it suffices to define $E(\SPhor(a,\inx K))\equiv (w_0=(K,\lambda\_.a))*J(K,\lambda\_.a)\leftrightarrow P\times (a=x_0)$ and $E(\SPhor(a,*))\equiv\colAt(a,a)*(P\times (a=x_0)^2)\leftrightarrow P\times (a=x_0)$, all cases are simple.
\end{proof}
Our main goal now will be proving that $D$ is contractible. To do that, we start by expressing it in terms of pushouts. By a flattening argument similar to the proof of theorem \ref{t:sp2set}, we have a pushout diagram:
    \begin{equation}
        \label{eq:stage1}
        \begin{tikzcd}
            (B\Sigma_3\hq B\Sigma_2)\times P \arrow[d] \arrow[r, "\proj_2"] & P \arrow[d] \\
            \LL \arrow[r] & D \arrow[ul, pos=0, phantom, "\ulcorner"]
        \end{tikzcd}
    \end{equation}
    where $\LL:\equiv\sum_{s:S^3_2(X)}Q(\SPinl(s))$. By a second flattening argument, using the equality proof of definition \ref{d:E} and commutativity of $\Sigma$-types for the base space, we have:
    \begin{equation}
        \label{eq:stage2}
        \begin{tikzcd}
            B\Sigma_2\times\rhobar \arrow[d] \arrow[r, "\proj_2"] & \rhobar \arrow[d] \\
            \Lam \arrow[r] & \LL \arrow[ul, pos=0, phantom, "\ulcorner"]
        \end{tikzcd}
    \end{equation}
    where $\rhobar:\equiv\sum_{a,b:X}Q(\SPinl(\Sinr(a,b)))$ and $\Lam:\equiv\sum_{u:X^3\hq \Sigma_3}((w_0=u)*J(u))$. Since pushouts commute with $\Sigma$-types, we can also write pushout diagrams:
    \begin{equation}
        \label{eq:stage3}
        \begin{tikzcd}
            \ColR\times P\arrow[r]\arrow[d] & P \arrow[d] & J(w_0) \arrow[r] \arrow[d] & \sum_{u:X^3\hq \Sigma_3}J(u) \arrow[d]\\
            \ColR \arrow[r] & \rhobar \arrow[ul, pos=0, phantom, "\ulcorner"] & \mathbf{1} \arrow[r] & \Lam \arrow[ul, pos=0, phantom, "\ulcorner"]
        \end{tikzcd}
    \end{equation}
    We can check $\left(\sum_{u:X^3\hq \Sigma_3}\proj_2u=\lambda\_.x_0\right)\simeq B\Sigma_3$ and $(P\times(\proj_2u=\lambda\_.x_0))\to\cb(u)$, so: 
    \begin{equation}
        \left(\sum_{u:X^3\hq \Sigma_3}\cb(u)\times P\times (\proj_2u=\lambda\_.x_0)\right)=\left(\sum_{u:X^3\hq \Sigma_3}P\times(\proj_2u=\lambda\_.x_0)\right)= P\times B\Sigma_3.
    \end{equation}
    Therefore, combining that with a final flattening argument, we get a pushout:
    \begin{equation}
    \label{eq:stage4}
        \begin{tikzcd}
            P\times B\Sigma_3 \arrow[r, "\id"] \arrow[d] & P\times B\Sigma_3 \arrow[d]\\
            \sum_{u:X^3\hq \Sigma_3}\cb(u)\arrow[r] & \sum_{u:X^3\hq \Sigma_3} J(u) \arrow[ul, pos=0, phantom, "\ulcorner"]
        \end{tikzcd}
    \end{equation}
    which, combined with corollary \ref{c:ColR}, gives us:
    \begin{equation}
    \label{eq:stage5}
        \sum_{u:X^3\hq \Sigma_3}\cb(u)\simeq\sum_{u:X^3\hq \Sigma_3} J(u)\simeq\ColR\times\sum_{u:X^3\hq \Sigma_3}\tr{w_0=u}.
    \end{equation}
To prove that $D$ is contractible, we have to take an inhabitant $d:D$ and show by induction on the pushouts above that there is a path between $d$ and the center of contraction. However in some cases, the inhabitant contains a witness to a case that makes $D$ automatically contractible, making the path trivial. In lemmas \ref{l:PtoD} and \ref{l:colToD}, we prove such contractions. The idea for the rest of the proof of contraction (in lemma \ref{l:Dcontr}) is that every case has such a witness or is trivial.
\begin{lemma}
    \label{l:PtoD}
    If $P$ holds, then $D$ is contractible.
\end{lemma}
\begin{proof}
    By path induction, assume $x_1\equiv x_2\equiv x_0$. Then $w_0=\col(x_0,x_0)$, so $\cb(w_0)$, $J(w_0)$ \linebreak and $\ColR$ all hold. Now we can prove $\sum_{u:X^3\hq \Sigma_3}\tr{w_0=u}\simeq B\Sigma_3$ by univalence, for the proof of $(K,\lambda\_.x_0,p)=_{\sum_{u:X^3\hq \Sigma_3}\tr{w_0=u}}(K,f,p')$ we notice that $f=\lambda\_.x_0$ by induction on $p'$. Therefore $\sum_{u:X^3\hq \Sigma_3}J(u)\simeq B\Sigma_3$. Then the right pushout of \eqref{eq:stage3} becomes:
    \begin{equation}
        \pushout{\mathbf{1}}{B\Sigma_3}{\mathbf{1}}{\Lam},
    \end{equation}
    so $\Lam\simeq B\Sigma_3$, by an equivalence $((K,f),p)\mapsto K$. By the left pushout, $\rhobar$ is contractible, so the diagram \eqref{eq:stage2} becomes:
    \begin{equation}
        \pushout{B\Sigma_2}{\mathbf{1}}{B\Sigma_3}{\LL}
    \end{equation}
    where the map $B\Sigma_2\to B\Sigma_3$ is the inclusion $K\mapsto K\sqcup\mathbf{1}$. This is the pushout definition of a mapping cone, so $L\simeq B\Sigma_3\hq B\Sigma_2$. The equivalence sends the $\rhobar$ leg to $*$, a $((K,f),p)$ in the $\Lam$ leg to $\inx K$, and the pushout glue to the mapping cone glue. Finally, the diagram \eqref{eq:stage1} becomes:
    \begin{equation}
        \pushout{B\Sigma_3\hq B\Sigma_2}{\mathbf{1}}{B\Sigma_3\hq B\Sigma_2}{D}
    \end{equation}
    where the vertical map is equal to the identity by function extensionality and mapping cone induction, since $\theta^x_K$ from definition \ref{d:SP3} is reflexivity on the first coordinate. Thus $D$ is contractible.
\end{proof}
To prove the next case, we introduce a helper type $W(u)$ and use the fact that it is a proposition for any $u$.
\begin{lemma}
    \label{l:Wprop}
    Assume $w_0=\col(x,y)$. Then $\WW:X^3\hq \Sigma_3\to\Type$ defined as:
    \begin{equation}
        \WW(u):\equiv\left(\sum_{K:B\Sigma_2}\Shor(K,x,y)=u\right)*\left((x=y)\times (\proj_2u=\lambda\_.x)\right)
    \end{equation}
    is proposition-valued.
\end{lemma}
\begin{proof}
    Assume for $u:X^3\hq \Sigma_3$ we have a $w:\WW(u)$, we prove $\WW(u)$ is contractible by induction on $w$, omitting the glue case. The right side of the join is a proposition, so if it is inhabited, the whole join is contractible. Now say we have a $(K,p):\sum_{K:B\Sigma_2}\Shor(K,x,y)=u$, by induction on $p$ we can assume $\Shor(K,x,y)\equiv u$. We set the center of contraction to $\inl(K,\refl)$, and we just need to find:
    \begin{equation}
        \prod_{(K',p'):\sum_{K':B\Sigma_2}\Shor(K',x,y)=\Shor(K,x,y)}\inl(K',p')=_{\WW(u)}\inl(K,\refl)
    \end{equation}
    Now $p'$ can be seen as an equivalence $e:K'\sqcup\mathbf{1}\simeq K\sqcup\mathbf{1}$ commuting with labelling functions. If $e(\inr*)=\inr*$, then $e$ restricts to a path $q:K'=K$ such that $s:\ap_{\Shor(\cdot,x,y)}(q)=p'$. We get the claim by generalizing over $q$ and $s$ and performing path induction on them. If $e(\inr *)=\inl k$ for some $k:K$, by the commutation relations, similarly to the proof of lemma \ref{l:coleq}, we get $x=y$. Then the right side of the join holds, so the join is contractible by fact \ref{f:joincontr} and we have a path.
\end{proof}
\begin{lemma}
    \label{l:colToD}
    Assume $w_0=\col(x,y)$. Then $D$ is contractible.
\end{lemma}
\begin{proof}
    For $(u,v):\sum_{u:X^3\hq \Sigma_3}\cb(u)$, we can prove that $W(u)$ holds. Indeed, $W(u)$ is a proposition, so we can assume the untruncated $\cb(u)$, which gives us $u=\col(x,y)$, which lies in the left side of the join for $K=\mathbf{2}$. Furthermore, $\cb(w_0)$, $\ColR$ and $J(w_0)$ all hold, so $\rhobar$ is contractible and $\Lam\simeq\sum_{u:X^3\hq \Sigma_3}J(u)$ by \eqref{eq:stage3}, the equivalence preserves the first coordinate. The vertical map of \eqref{eq:stage2} is then the map $F:B\Sigma_2\to \sum_{u:X^3\hq \Sigma_3}\cb(u)$ given by $K\mapsto(\Shor(K,x,y),\dots)$. Composing that with the contraction of $\rhobar$ centered at the point $(x,y,\dots)$, we get an equivalence of pushout diagrams, presenting the space $L$ as the mapping cone of $F$. We will now prove the contractibility of $D$, choosing the vertex of $L$ as the center. Let $d:D$, proceed by pushout induction on $d$ as an element of the pushout in \eqref{eq:stage1}. The only nontrivial case is when $d$ comes from $L$, because otherwise $P$ holds and by lemma \ref{l:PtoD} $D$ is contractible, so we have a path. Now we need to find a path
    \begin{equation}
        p:\prod_{t:\sum_{u:X^3\hq \Sigma_3} \cb(u)}\left(\inx_L(\inx t)=_D\inx_L*\right)
    \end{equation}
    such that $p(F(K))$ is the gluing path. For a given $t=(u,v)$, we proceed by pushout induction on the inhabitant of $W(u)$. Once again, the only nontrivial case is the left one, since otherwise $P$ holds. We therefore have a given $K$ and a path $\Shor(K,x,y)=u$. Since $\cb(u)$ is a proposition, we have a path $q:F(K)=t$. By path induction on $q$, we can assume $F(K)\equiv t$ and define $p$ as the gluing path on $F(K)$.
\end{proof}
\begin{lemma}
    \label{l:Dcontr}
    The type $D$ is contractible.
\end{lemma}
\begin{proof}
    Express $D$ as a pushout resulting from diagrams \eqref{eq:stage1}, \eqref{eq:stage2} and \eqref{eq:stage3}. The center of contraction will then be $c:\equiv \inx_L(\inx_{\Lambda}(\inx_{\mathbf{1}}*))$. We pick a $d:D$ and find paths $c=d$ by iterated pushout induction:
    \begin{itemize}
        \item For the $P$ and the glue case, since $P$ holds, the claim follows by lemma \ref{l:PtoD}.
        \item Let $d=\inx_Ll$.
        \begin{itemize}
            \item Let $l=\inx_\rhobar r$. Once again the $P$ and the glue case are easy. In the $\ColR$ case, we have $(x,y,c):\ColR$, and once again we can prove $D$ is contractible by eliminating the truncation of $c$ and using lemma \ref{l:colToD}. The glue case for $l$ is analogous, since it also makes $\rhobar$ hold.
            \item Let $l=\inx_\Lam L$. 
            \begin{itemize}
                \item Let $L=\inx_{\sum_{u:X^3\hq \Sigma_3}J(u)}(u,p)$. Then $p:\cb(u)*(P\times \proj_2u=\lambda\_.x_0)$. Once again the right and glue case are easy, since $P$ holds. In the left case, we can prove $D$ is contractible by eliminating the truncation from $\cb(u)$ and using lemma \ref{l:colToD}. The glue case for $L$ is analogous, since it makes $J(w_0)$ hold. 
                \item Let $L=\inx_{\mathbf{1}}*$. Then the path is $\refl$.
            \end{itemize}
        \end{itemize}
    \end{itemize}
    Since in all cases $D$ is contractible or the path is $\refl$, we have constructed a uniform path $\prod_{d:D}c=d$.
\end{proof}
As in the proof for $\SP^2$, all that is left is to apply the fundamental theorem of identity types and check that $Q(z_0)=(z_0=z_0)$ is indeed contractible.
\begin{lemma}
    \label{l:Qz0contr}
    The type $Q(z_0)\equiv (w_0=w_0)*(J(w_0))$ is contractible.
\end{lemma}
\begin{proof}
    The type $Q(z_0)$ is inhabited by $\inl\refl$. Notice $J(w_0)=\cb(w_0)*P$. Since it is a proposition, by fact \ref{f:joincontr} it suffices to prove that for $h:w_0=w_0$, $\inl\refl=_{Q(z_0)}\inl h$. We can interpret $h$ as an equivalence $e:\mathbf{3}\simeq\mathbf{3}$ commuting with the labelling functions. There are three kinds of equivalences, we consider cases based on them.
    \begin{itemize}
        \item If $e$ is an identity, $h=\refl$.
        \item If $e$ is a 3-cycle, by the commutation with labelling we get $x_0=x_1=x_2$, which gives us $P$ and $J(w_0)$.
        \item If $e$ is a transposition, say for $a,b,c:\mathbf{3}$ we have $e(a)=b,e(b)=a,e(c)=c$ and $x_m:\equiv (\proj_2w_0)(m)$ for $m=a,b,c$. Then by the commutation relation we have $x_a=x_b$ and $\lVert w_0=\linebreak\col(x_a,x_c)\rVert$ holds, which gives us $\cb(w_0)$ and $J(w_0)$.\qedhere
    \end{itemize}
\end{proof}
Now we can finish the proof of theorem \ref{t:sp3set}.
\begin{proof}
    By the fundamental theorem of identity types, since $\sum_{z:\SP^3X}Q(z)$ is contractible by lemma \ref{l:Dcontr}, the loop space $z_0=_{\SP^3X}z_0$ is equal to $Q(z_0)$. Then, by lemma \ref{l:Qz0contr}, it is contractible.
\end{proof}

\end{document}